\documentclass[leqno,a4paper,10pt]{amsart}
\usepackage{amsmath,amsthm,amssymb,amsfonts,enumerate,color}
\allowdisplaybreaks[4]

\newtheorem{thm}{Theorem}[section]
\newtheorem{prop}[thm]{Proposition}
\newtheorem{cor}[thm]{Corollary}
\newtheorem{lem}[thm]{Lemma}
\newtheorem{defi}[thm]{Definition}
\newtheorem{remark}[thm]{Remark}
\newtheorem{example}[thm]{Example}
\newtheorem{pb}[thm]{Problem}

\newtheorem{theoA}{Theorem}

\newtheorem{remA}[theoA]{Remark}

\newenvironment{definition}{\begin{defi}\rm}{\end{defi}}

\numberwithin{equation}{section}
\newcommand{\R}{{\mathbb R}}

\newcommand{\norm}[1]{\left\Vert#1\right\Vert}
\newcommand{\abs}[1]{\left\vert#1\right\vert}

\newcommand{\pos}[1]{\left[#1\right]_{+}}

\DeclareMathOperator{\essinf}{ess\,inf}
\newcommand{\BMO}{\text{\rm BMO}}
\newcommand{\BLO}{\text{\rm BLO}}

\begin{document}

\title[Heat semigroup characterization of BLO spaces]
{New characterizations of BLO spaces by heat semigroups  and   applications}


\author{Shaohong Liang, Dongyong Yang and Chao Zhang$^\dag$ }

 \address{Shaohong Liang \\Department of Mathematics,\ Zhejiang University of Science and Technology,\ Hangzhou 310023,  P. R. China}
 \email{liangshaohongzk@163.com}

\smallskip

\address{Dongyong Yang \\School of Mathematical Sciences,\ Xiamen University,\ Xiamen 361005,  P. R. China}
 \email{dyyang@xmu.edu.cn}

\smallskip

 \address{Chao Zhang  \\Department of Mathematics,\ Zhejiang University of Science and Technology,\ Hangzhou 310023,  P. R. China}
 \email{zaoyangzhangchao@163.com}

\thanks{The second author is supported by the National Natural Science Foundation of China (No. 12571106) and the Natural Science Foundation of Fujian Province of China(No. 2025J01029). The third author is supported  by the National Natural Science Foundation of China (Grant No.  11971431) and the Natural Science Foundation of Zhejiang Province (Grant No. LY22A010011)}

\thanks{$^\dag$ Corresponding Author.  }

\date{}
\maketitle

\begin{abstract}
In this paper, we give two new characterizations of the bounded lower oscillation (BLO) space by using the Gaussian heat semigroup. Using these characterizations, we prove the boundary lower-envelope estimate of the solutions of the heat equation with BLO boundary data,   a Gaussian John-Nirenberg criterion and a quantitative estimates of the BLO-norm. Also, we reprove the BMO-BLO boundedness of the Littlewood-Paley $g$-function by using the semigroup method.

 \end{abstract}


\vskip 1cm

\section{Introduction} \label{Sec:L2}
The theory of Hardy spaces and BMO spaces on the Euclidean space $\R^n$ plays an important role in various fields of analysis and partial differential equations. It is well known that $H^1(\R^n)$  and BMO($\R^n$) are, respectively, suitable substitutes of $L^1(\R^n)$ and $L^\infty(\R^n)$. For example, it is known that the classical Riesz transforms are not bounded on $L^1(\R^n)$ and $L^\infty(\R^n)$, but bounded on $H^1(\R^n)$  and BMO($\R^n$). In 1980, in order to study the relationship between $A_p$ weights and $\BMO$ functions,  R. Coifman and R. Rochberg in \cite{CR} introduced BLO($\R^n$) space which is a subset of BMO($\R^n$), and   discussed the
relationship between BMO functions and BLO functions. We recall the definition of BLO space as follows.
 \begin{definition}\label{Def:BLO}
 The space ${\rm BLO}(\R^n)$ consists of locally integrable functions $f$ such that
$$\norm{f}_{{\rm BLO}(\R^n)}:=\sup_B {1\over |B|} \int_B [f(x)-\underset{ B} {\essinf} \ f]dx<\infty,$$
where the supremum is taken over all balls $B\subset \R^n$.
 \end{definition}

 By the definition of $\rm{BLO}(\R^n)$ space, we know that it is not a linear space. For example, $f(x)=-\ln |x| \cdot \chi_{\{|x|>0\}}(x)\in \rm{BLO}(\R)$, but $-f\notin {\rm BLO}(\mathbb R)$ since $-f(x)=\ln|x|$ is unbounded below near zero. In fact, we will give an exact proof of this in Section \ref{Sec:4}.
 In \cite{Bennett}, C. Bennett  established another characterization of BLO($\R^n$) space with a variant Hardy-Littlewood maximal function. In recent years, there has been increasing interest in the study of BLO type spaces. Y. Jiang \cite{Jiang} introduced  RBLO space for non-doubling measures $\mu$ and established boundedness of a kind of maximal singular integral operator and the Hardy-Littlewood maximal operator from $L^\infty(\mu)$ to RBLO($\mu)$. L. Tang \cite{Tang} introduced BLO spaces associated with the sections and studied the BLO type estimates for a kind of maximal singular integral operator and the Hardy-Littlewood maximal operator. D. Yang et al \cite{YYZ} studied the space of BLO type associated with Schr\"odinger operators. In \cite{AlmeidaBetancor}, V. Almeida et al. introduced the BLO spaces in the rational Dunkel setting.  The authors in \cite{HWYZ} considered the BMO-BLO boundedness of the differential transforms. Karagulyan proved the boundedness of  the maximal operator on spaces BMO and BLO on an abstract measure space equipped with a ball-basis in \cite{Karagulyan}. For more information and recent development about the BLO spaces, see, for example, \cite{LLW, Popoli, ZL}.

 The heat semigroup associated with the  Laplacian is a basic tool in harmonic analysis.  We use the Gaussian heat kernel
\[
 W_t(x-y)=(4\pi t)^{-n/2}
 \exp\!\left(-\frac{|x-y|^2}{4t}\right),
 \qquad t>0.
\]
Let
\begin{equation}\label{eq:theta}
 \vartheta_n:=\int_{B(0,1)}(4\pi)^{-n/2}e^{-|z|^2/4}\,dz.
\end{equation}
The change of variables $y=x+\sqrt t\,z$ shows that
\[
 \frac1{\vartheta_n}\int_{B(x,\sqrt t)}W_t(x-y)\,dy=1.
\]

Using the heat semigroup and Poisson semigroup to characterize function spaces can be traced to the work of E. Stein and M. Taibleson; see  \cite{SteinTopic, Taibleson1,Taibleson2, Taibleson3}.  In this paper, we will give a new characterization of ${\rm BLO}(\R^n)$ by using the heat semigroup  as follows.

\begin{thm}\label{Thm:heatC1.2}
For $f\in L^1_{\mathrm{loc}}(\R^n)$, the following statements are equivalent:
\begin{enumerate}
 \item[(i)] $f\in\BLO(\R^n)$;
 \item[(ii)]
 \begin{equation}\label{eq:full-condition}
  \sup_{x\in\R^n,\,t>0}
  \int_{\R^n}W_t(x-y)
  \pos{f(y)-\underset{{z\in B(x,\sqrt t)}}\essinf f(z)}\,dy<\infty;
 \end{equation}
 \item[(iii)]
 \begin{equation}\label{eq:local-condition}
  \sup_{x\in\R^n,\,t>0}
  \left[
   \frac1{\vartheta_n}\int_{B(x,\sqrt t)}W_t(x-y)f(y)\,dy
   -\underset{{z\in B(x,\sqrt t)}}\essinf f(z)
  \right]<\infty.
 \end{equation}
\end{enumerate}
Moreover, there are constants $0<c_n\leq C_n<\infty$, depending only on the
dimension, such that
\begin{align}\label{eq:full-equivalence}
 c_n\norm{f}_{\BLO(\R^n)}
 &\leq
 \sup_{x\in\R^n,\,t>0}
 \int_{\R^n}W_t(x-y)
 \pos{f(y)-\underset{{z\in B(x,\sqrt t)}}\essinf f(z)}\,dy
 \leq C_n\norm{f}_{\BLO(\R^n)},
\end{align}
and
\begin{align}\label{eq:local-equivalence}
 e^{-1/4}\norm{f}_{\BLO(\R^n)}
 &\leq
 \sup_{x\in\R^n,\,t>0}
 \left[
  \frac1{\vartheta_n}\int_{B(x,\sqrt t)}W_t(x-y)f(y)\,dy
  -\underset{{z\in B(x,\sqrt t)}}\essinf f(z)
 \right]\leq e^{1/4}\norm{f}_{\BLO(\R^n)}.
\end{align}
\end{thm}

The positive part and the comparison with the boundary-data infimum are
essential. Consider the symmetric expression
\begin{equation}\label{eq:symmetric-functional}
 \sup_{t>0}
 \left\|W_tf(\cdot)-
 \inf_{z\in B(\cdot,\sqrt t)}W_tf(z)\right\|_{L^\infty(\R^n)}.
\end{equation}
Then, we have the following result for functions in $\BLO(\R^n)$ spaces.

\begin{thm}\label{Thm:heatC}
For every function
$f\in {\rm BLO}(\R^n)$, there exists a constant $C>0$ such that  $$\sup_{t>0}\norm{W_t f(\cdot)-\underset{z\in B(\cdot,\sqrt t)} {\inf} W_tf(z)}_{L^\infty(\R^n)}\le C \norm{f}_{\BLO(\R^n)}. $$
\end{thm}

As we have mentioned before, the space ${\rm BLO}(\R^n)$  is also related with the $A_1$ Muckenhoupt weights. We recall the definition of the $A_1$ Muckenhoupt weights as follows.

\begin{definition}($A_1$ Muckenhoupt Weights)
A non-negative, locally integrable function $w: \mathbb{R}^n \to \mathbb{R}$ is said to be an $A_1$ weight if there exists a constant $C > 0$ such that
\begin{equation}\label{eq:1.11}
M w(x) \leq C w(x) \quad \text{for almost every } x \in \mathbb{R}^n,
\end{equation}
where $M$ denotes the Hardy--Littlewood maximal operator. The smallest  constant $C$, such that the above inequality holds, is called the $A_1$ constant of $w$.
\end{definition}

In \cite{CR}, R. Coifman and R. Rochberg proved a characterization of BLO space as follows.

  \begin{thm}
   A function $f$ is in $\mathrm{BLO}(\R^n)$ if and only if $e^{\varepsilon f}\in A_1$ for some positive constant $\varepsilon$.
   \end{thm}
Based on the theorem above and the heat characterization of BLO space as in Theorem \ref{Thm:heatC1.2}, we can get another new weight characterization of BLO space by using the heat semigroup.

\begin{thm}[Heat Characterization via $A_1$ Weights]\label{HeatCbyWeight}
For a locally integrable function $f: \mathbb{R}^n \to \mathbb{R}$, the following three statements are equivalent:
\begin{enumerate}
    \item[(i)] $f \in \mathrm{BLO}(\mathbb{R}^n)$.
    \item[(ii)] There exists $\varepsilon > 0$ such that $e^{\varepsilon f} \in A_1$.
    \item[(iii)] There exists $\varepsilon > 0$ and $C > 0$ such that for all $t > 0$ and almost every $x \in \mathbb{R}^n$,
    \[
    W_t(e^{\varepsilon f})(x) \leq C e^{\varepsilon f(x)}.
    \]
\end{enumerate}
\end{thm}

This article is organized as follows. In Section \ref{Sec:heat}, we give the proofs of Theorem \ref{Thm:heatC1.2} and some applications of the theorem. In Section \ref{Sec:1.3}, we give the proofs of Theorem \ref{Thm:heatC}. We also prove the regularity of solutions of the heat equation with BLO boundary values using Theorem \ref{Thm:heatC}. In Section \ref{Sec3}, we give the proof of Theorem \ref{HeatCbyWeight}. Finally, in Section \ref{Sec:4}, we give an example showing that $\rm{BLO}$ is not a linear space, although it has a stability property under $L^\infty$ perturbations.

Throughout this article, the letters $C, c$ will denote  positive
constants which may change from one instance to another and depend on
the parameters involved. We will make  frequent use, without
mentioning it in relevant places, of the fact that for a positive
$A$ and a non-negative constant $a,$
$$\sup\limits_{t>0}t^a
e^{-At}=C_{a, A}<\infty.$$

\section{Proof of the heat semigroup characterization and its applications}\label{Sec:heat}
In this section, we will give the proof of Theorem \ref{Thm:heatC1.2} which are heat characterizations of $\BLO$ spaces.  We also give the applications of the heat characterizations in the boundary value problem, Gaussian John-Nirenberg criterion, nonlinear stability of BLO spaces, and  the BMO--BLO boundedness of the Littlewood--Paley $g$-function.

\subsection{Proof of Theorem \ref{Thm:heatC1.2}}

Let us prove Theorem \ref{Thm:heatC1.2} first.
\begin{proof}[Proof of Theorem~\ref{Thm:heatC1.2}]
Assume first that $f\in\BLO(\R^n)$. Fix $x\in\R^n$ and $t>0$, and put
$B=B(x,\sqrt t)$. Since $f-\underset B\essinf f\geq0$ almost everywhere on $B$,
\begin{equation}\label{eq:near-upper}
 \int_BW_t(x-y)\pos{f(y)-\underset B\essinf f}\,dy
 \leq C_n\norm{f}_{\BLO(\R^n)}.
\end{equation}

For $k\geq0$, let $A_k=2^{k+1}B\setminus2^kB$. If $y\in A_k$, then
\[
 W_t(x-y)\leq C_nt^{-n/2}e^{-4^{k-1}}.
\]
Since $\essinf_{2^{k+1}B}f\leq \underset B\essinf f$, we have
\[
 \pos{f(y)-\underset B\essinf f}
 \leq f(y)-\underset{2^{k+1} B}\essinf f
 \quad\text{for almost every }y\in2^{k+1}B.
\]
Therefore
\begin{align*}
 &\int_{A_k}W_t(x-y)\pos{f(y)-\underset B\essinf f}\,dy\\
 &\quad\leq C_nt^{-n/2}e^{-4^{k-1}}
 \int_{2^{k+1}B}\left(f(y)-\underset {z\in 2^{k+1}B}\essinf f(z)\right)dy\\
 &\quad\leq C_ne^{-4^{k-1}}2^{n(k+1)}\norm{f}_{\BLO(\R^n)}.
\end{align*}
Summing in $k$ and using \eqref{eq:near-upper} proves the upper inequality in
\eqref{eq:full-equivalence}.

Conversely, for $y\in B(x,\sqrt t)$,
$W_t(x-y)\geq(4\pi t)^{-n/2}e^{-1/4}$. Hence
\begin{align*}
 &\int_{\R^n}W_t(x-y)
 \pos{f(y)-\essinf_{z\in B(x,\sqrt t)}f(z)}\,dy\\
 &\quad\geq (4\pi t)^{-n/2}e^{-1/4}
 \int_{B(x,\sqrt t)}
 \left(f(y)-\underset {z\in B(x,\sqrt t)}\essinf   f(z)\right)dy\\
 &\quad=a_n\frac1{|B(x,\sqrt t)|}
 \int_{B(x,\sqrt t)}
 \left(f(y)-\underset {z\in B(x,\sqrt t)}\essinf f(z)\right)dy,
\end{align*}
where $a_n=(4\pi)^{-n/2}e^{-1/4}|B(0,1)|>0$. Taking the supremum over $x$
and $t$ proves the lower inequality in \eqref{eq:full-equivalence}.

On $B(x,\sqrt t)$, the normalized density satisfies
\begin{equation}\label{eq:normalized-density}
 e^{-1/4}\frac{\mathbf1_{B(x,\sqrt t)}(y)}{|B(x,\sqrt t)|}
 \leq\frac{W_t(x-y)\mathbf1_{B(x,\sqrt t)}(y)}{\vartheta_n}
 \leq e^{1/4}\frac{\mathbf1_{B(x,\sqrt t)}(y)}{|B(x,\sqrt t)|}.
\end{equation}
Apply this to
$f(y)-\underset {z\in B(x,\sqrt t)}\essinf  f(z)\geq0$ on the ball and take the
supremum over $x$ and $t$. This gives \eqref{eq:local-equivalence} and
completes the proof.
\end{proof}

\begin{remark}
The localized proof only uses two-sided bounds for the kernel on its natural
ball. The full-tail proof additionally uses the rapid Gaussian decay.
\end{remark}
\subsection{A boundary lower-envelope estimate}
Using the heat characterization in Theorem \ref{Thm:heatC1.2}, we can get a boundary lower-envelope estimate of the heat equation as follows.
\begin{cor}\label{cor:boundary-envelope}
Let $f\in\BLO(\R^n)$. Consider the solution $u(x,t)=W_tf(x)$ to the heat equation
\begin{align*}
\begin{cases}
\partial_t u = \Delta u, & \text{in } \mathbb{R}^n \times (0,\infty) \\
u(x,0) = f(x), & \text{on } \mathbb{R}^n.
\end{cases}
\end{align*} Then
\begin{equation}\label{eq:boundary-envelope}
 u(x,t)\leq \underset {z\in B(x,\sqrt t)}\essinf  f(z)
 +C_n\norm{f}_{\BLO(\R^n)}.
\end{equation}
Consequently, at every Lebesgue point $x$ of $f$,
\begin{equation}\label{eq:heat-above-boundary}
 \sup_{t>0}(W_tf(x)-f(x))
 \leq C_n\norm{f}_{\BLO(\R^n)}.
\end{equation}
\end{cor}

\begin{proof}
Since the heat kernel has mass one,
\begin{align*}
 &W_tf(x)-\underset {z\in B(x,\sqrt t)}\essinf  f(z)\\
 &\quad=\int_{\R^n}W_t(x-y)
 \left(f(y)-\underset{z\in B(x,\sqrt t)}\essinf f(z)\right)dy\\
 &\quad\leq\int_{\R^n}W_t(x-y)
 \pos{f(y)-\underset{z\in B(x,\sqrt t)}\essinf f(z)}\,dy.
\end{align*}
Apply the upper inequality in \eqref{eq:full-equivalence}. At a Lebesgue
point $x$, $\underset{z\in B(x,r)}\essinf f(z)\leq f(x)$ for every $r>0$, which gives
\eqref{eq:heat-above-boundary}.
\end{proof}

\subsection{A Gaussian John--Nirenberg criterion}
We can also develop a Gaussian John-Nirenberg criterion by using the new heat characterization.
\begin{prop}\label{prop:Gaussian-JN}
There are constants $\alpha_n>0$ and $A_n>1$ such that every nonconstant
$f\in\BLO(\R^n)$ satisfies
\begin{equation}\label{eq:Gaussian-JN}
 \sup_{x\in\R^n,\,t>0}
 \frac1{\vartheta_n}\int_{B(x,\sqrt t)}W_t(x-y)
 \exp\!\left(
  \frac{\alpha_n
  (f(y)-\underset{z\in B(x,\sqrt t)}\essinf f(z))}{\norm{f}_{\BLO(\R^n)}}
 \right)dy
 \leq A_n.
\end{equation}
Conversely, if for some $\lambda>0$ and $A\geq1$,
\begin{equation}\label{eq:Gaussian-JN-converse}
 \sup_{x\in\R^n,\,t>0}
 \frac1{\vartheta_n}\int_{B(x,\sqrt t)}W_t(x-y)
 e^{\lambda(f(y)-\underset{z\in B(x,\sqrt t)}\essinf f(z))}\,dy
 \leq A,
\end{equation}
then $f\in\BLO(\R^n)$ and
\begin{equation}\label{eq:Gaussian-JN-norm}
 \norm{f}_{\BLO(\R^n)}
 \leq\frac{e^{1/4}}{\lambda}\log A.
\end{equation}
\end{prop}

\begin{proof}
Put $N:=\norm{f}_{\BLO(\R^n)}>0$. We have $\norm{f}_{\BMO(\R^n)}\leq2N$ and
\[
 f_B-\underset{z\in B}\essinf  f(z)\leq N.
\]
The integral John--Nirenberg inequality for  BLO spaces (see \cite[Lemma 2.1]{WZT18})  gives dimensional constants
$\alpha_n,A_n'$ such that
\[
 \frac1{|B|}\int_B
 \exp\!\left(
  \frac{\alpha_n(f(y)-\underset{z\in B }\essinf f(z))}{N}
 \right)dy\leq A_n'.
\]
The upper bound in \eqref{eq:normalized-density} proves
\eqref{eq:Gaussian-JN}.

Conversely, Jensen's inequality applied to the normalized localized Gaussian
average yields
\begin{align*}
 &\exp\!\left\{\lambda\left[
  \frac1{\vartheta_n}\int_{B(x,\sqrt t)}W_t(x-y)f(y)\,dy
  -\underset{z\in B(x,\sqrt t)}\essinf f(z)
 \right]\right\}\\
 &\quad\leq\frac1{\vartheta_n}\int_{B(x,\sqrt t)}W_t(x-y)
 e^{\lambda(f(y)-\underset{z\in B(x,\sqrt t)}\essinf f(z))}\,dy.
\end{align*}
Use \eqref{eq:Gaussian-JN-converse} and the lower inequality in
\eqref{eq:local-equivalence} to obtain \eqref{eq:Gaussian-JN-norm}.
\end{proof}

\subsection{Nonlinear stability}
We list a nonlinear stability of BLO spaces at here.
\begin{prop}\label{prop:nonlinear-stability}
Let $f\in\BLO(\R^n)$.
\begin{enumerate}
 \item[(i)] If $h\in L^\infty(\R^n)$, then
 \begin{align}\label{eq:bounded-perturbation-functional}
 &\sup_{x\in\R^n,\,t>0}\int_{\R^n}W_t(x-y)
 \pos{f(y)+h(y)-\underset{z\in B(x,\sqrt t)}\essinf (f(z)+h(z))}\,dy\nonumber \\
 &\quad\leq
 \sup_{x\in\R^n,\,t>0}\int_{\R^n}W_t(x-y)
 \pos{f(y)-\underset{z\in B(x,\sqrt t)}\essinf f(z)}\,dy
 +2\norm{h}_{L^\infty(\R^n)}.
 \end{align}
 In particular,
 \[
  \norm{f+h}_{\BLO(\R^n)}
  \leq\norm{f}_{\BLO(\R^n)}+2\norm{h}_{L^\infty(\R^n)}.
 \]
 \item[(ii)] If $\Phi:\R\to\R$ is nondecreasing and $L$-Lipschitz, then
 \begin{align}\label{eq:Lipschitz-functional}
 &\sup_{x\in\R^n,\,t>0}\int_{\R^n}W_t(x-y)
 \pos{\Phi(f(y))-\underset{z\in B(x,\sqrt t)}\essinf \Phi(f(z))}\,dy\nonumber \\
 &\quad\leq L
 \sup_{x\in\R^n,\,t>0}\int_{\R^n}W_t(x-y)
 \pos{f(y)-\underset{z\in B(x,\sqrt t)}\essinf f(z)}\,dy.
 \end{align}
 Consequently, $\Phi\circ f\in\BLO(\R^n)$. In particular, lower and
 two-sided truncations preserve $\BLO(\R^n)$.
\end{enumerate}
\end{prop}

\begin{proof}
For every ball $B$,
\[
 \essinf_{z\in B}(f(z)+h(z))
 \geq \underset{z\in B}\essinf f(z)-\norm{h}_{L^\infty(\R^n)}.
\]
It follows that
\[
 \pos{f+h-\underset{z\in B }\essinf  (f+h)}
 \leq\pos{f-\underset{z\in B }\essinf f}+2\norm{h}_{L^\infty(\R^n)}.
\]
Integration proves \eqref{eq:bounded-perturbation-functional}, and Definition
\ref{Def:BLO} gives the displayed BLO estimate.

Continuity and monotonicity give
\[
\underset{z\in B}\essinf\Phi(f(z))=\Phi\!\left(\underset{z\in B}\essinf f(z)\right).
\]
The Lipschitz property then yields
\[
 \pos{\Phi(f(y))-\Phi(\underset{z\in B}\essinf f)}
 \leq L\pos{f(y)-\underset{z\in B}\essinf f}.
\]
Integration proves \eqref{eq:Lipschitz-functional}; apply
Theorem~\ref{Thm:heatC1.2}.
\end{proof}

\subsection{BMO--BLO boundedness of the  Littlewood-Paley $g$-function}
In \cite{MY}, Meng and Yang proved that the Littlewood--Paley $g$-function is bounded from BMO($\R^n$) to BLO($\R^n$). Here, as an application of Theorem \ref{Thm:heatC1.2}, we reprove this result using our new heat characterization of BLO spaces.
\begin{thm}\label{Thm:gbound}
Let $g$ be the Littlewood-Paley $g$-function defined by
\[
g(f)(x) = \left( \int_0^\infty |t\partial_t W_t f(x)|^2 \frac{dt}{t} \right)^{1/2}.
\]
Then, for any $f \in \mathrm{BMO}(\mathbb{R}^n)$ such that $g(f)(x)<\infty$ a.e. $x\in \R^n$, the function $[g(f)]^2$ belongs to $\mathrm{BLO}(\mathbb{R}^n)$ with the norm estimate
\[
\|[g(f)]^2\|_{\mathrm{BLO}(\mathbb{R}^n)} \leq C \|f\|_{\mathrm{BMO}(\mathbb{R}^n)}^2,
\]
where $C > 0$ depends only on the dimension $n$.
\end{thm}
\begin{proof}
Set
\[
A_sf(y)=s\partial_sW_sf(y),
\qquad h(y)=[g(f)(y)]^2.
\]
Fix $x\in\R^n$ and $r>0$, and put $B=B(x,r)$. Split
\[
h(y)=N_r(y)+F_r(y),
\]
where
\[
N_r(y):=\int_0^{4r^2}|A_sf(y)|^2\frac{ds}{s},
\qquad
F_r(y):=\int_{4r^2}^{\infty}|A_sf(y)|^2\frac{ds}{s}.
\]

We first establish the local square-function estimate
\begin{equation}\label{eq:local-square-estimate}
\frac1{|B|}\int_BN_r(y)\,dy
\leq C\|f\|_{\mathrm{BMO}(\R^n)}^2.
\end{equation}
Decompose
\[
f=f_1+f_2+f_{8B},
\qquad
f_1=(f-f_{8B})\chi_{8B},
\quad
f_2=(f-f_{8B})\chi_{\R^n\setminus8B}.
\]
Since $A_s$ annihilates constants, the $L^2$ boundedness of the
Littlewood--Paley $g$-function and the John--Nirenberg inequality give
\[
\int_B\int_0^{4r^2}|A_sf_1(y)|^2\frac{ds}{s}\,dy
\leq C\|f_1\|_{L^2(\R^n)}^2
\leq C|B|\|f\|_{\mathrm{BMO}(\R^n)}^2.
\]
For $y\in B$, decompose $\R^n\setminus8B$ into annuli on which
$|y-z|\sim2^kr$. The heat-kernel derivative estimate and the
John--Nirenberg inequality imply, for $0<s\leq4r^2$,
\begin{align*}
|A_sf_2(y)|
&\leq C\|f\|_{\mathrm{BMO}(\R^n)}
 \sum_{k\geq0}(k+1)\left(\frac{4^kr^2}{s}\right)^{n/2}\left(1+\frac{4^kr^2}{s}\right)e^{-c\frac{4^kr^2}{s}},
 \\
&\leq C\|f\|_{\mathrm{BMO}(\R^n)}
 \left(\frac{s}{r^2}\right)^{1/2}.
\end{align*}
Here we used $u_k\geq4^{k-1}$ and absorbed the polynomial factors into
the exponential decay. Consequently,
\[
\int_0^{4r^2}|A_sf_2(y)|^2\frac{ds}{s}
\leq C\|f\|_{\mathrm{BMO}(\R^n)}^2,
\qquad y\in B.
\]
Combining the last two estimates proves \eqref{eq:local-square-estimate}.

To control the tail $F_r$, we use the following standard estimate.
\begin{lem}\label{lem:BMOest}
\begin{enumerate}
\item
For $f \in \mathrm{BMO}(\mathbb{R}^n)$ and $s > 0$, there exists a constant $C > 0$ depending only on $n$ such that:
\[
|s\partial_s W_s f(y)| \leq C \|f\|_{\mathrm{BMO}(\mathbb{R}^n)} \quad \text{for all } y \in \mathbb{R}^n.
\]
\item For $f \in \mathrm{BMO}(\mathbb{R}^n)$, $s > 0$, and $x, z \in \mathbb{R}^n$, we have
\[
\left| |s\partial_s W_s f(x)|^2 - |s\partial_s W_s f(z)|^2 \right| \leq C \|f\|_{\mathrm{BMO}(\mathbb{R}^n)}^2 \frac{|x - z|}{\sqrt{s}}.
\]
\end{enumerate}
\end{lem}

\begin{proof}
${\it (1)}$ We provide a complete proof using the kernel representation and careful estimates.
The time derivative has the kernel representation
\[
s\partial_s W_s f(y) = \int_{\mathbb{R}^n} K_s(y,z) f(z)  dz,
\]
where
\[
K_s(y,z) = (4\pi s)^{-n/2} e^{-\frac{|y-z|^2}{4s}} \left[ -\frac{n}{2} + \frac{|y-z|^2}{4s} \right].
\]
We know that
\[
\int_{\mathbb{R}^n} K_s(y,z)  dz = 0.
\]
Let $B = B(y, \sqrt{s})$. Then
\[
s\partial_s W_s f(y) = \int_{\mathbb{R}^n} K_s(y,z) [f(z) - f_B]  dz.
\]
Decompose the integral into two parts
\[
s\partial_s W_s f(y) = \tilde I_{\text{near}}(y) + \tilde I_{\text{far}}(y),
\]
where
\begin{align*}
\tilde I_{\text{near}}(y) &= \int_{2B} K_s(y,z) [f(z) - f_B]  dz, \\
\tilde I_{\text{far}}(y) &= \int_{\mathbb{R}^n \setminus 2B} K_s(y,z) [f(z) - f_B]  dz.
\end{align*}
We estimate the near part first.
For $z \in 2B$, we have $|y-z| \leq 2\sqrt{s}$. And
\[
|K_s(y,z)| \leq C s^{-n/2} \left(1 + \frac{|y-z|^2}{s}\right) e^{-\frac{|y-z|^2}{4s}} \leq C s^{-n/2}.
\]
Therefore,
\[
|\tilde I_{\text{near}}(y)| \leq C s^{-n/2} \int_{2B} |f(z) - f_B|  dz.
\]
By the John-Nirenberg inequality, we get
\[
\int_{2B} |f(z) - f_B|  dz \leq C |2B| \|f\|_{\mathrm{BMO}(\R^n)} = C s^{n/2} \|f\|_{\mathrm{BMO}(\R^n)}.
\]
Thus,
\[
|\tilde I_{\text{near}}(y)| \leq C s^{-n/2} \cdot C s^{n/2} \|f\|_{\mathrm{BMO}(\R^n)} = C \|f\|_{\mathrm{BMO}(\R^n)}.
\]

For $z \in \mathbb{R}^n \setminus 2B$, we have $|y-z| \geq 2\sqrt{s}$. The kernel has rapid decay. Then, we have
\[
|K_s(y,z)| \leq C s^{-n/2} \left(\frac{|y-z|}{\sqrt{s}}\right)^2 e^{-\frac{|y-z|^2}{4s}} \leq C s^{-n/2} e^{-c\frac{|y-z|^2}{s}}.
\]
Decompose the far region $\mathbb{R}^n \setminus 2B$ into annuli
\[
A_k = \{z : 2^k \cdot 2\sqrt{s} \leq |y-z| < 2^{k+1} \cdot 2\sqrt{s}\}, \quad k \geq 0.
\]
For $z \in A_k$, we have $|y-z| \sim 2^k \sqrt{s}$. So,
\[
|K_s(y,z)| \leq C s^{-n/2} e^{-c 4^k}.
\]
Then, we have
\[
\left| \int_{A_k} K_s(y,z) [f(z) - f_B]  dz \right| \leq C s^{-n/2} e^{-c 4^k} \int_{A_k} |f(z) - f_B|  dz\le C e^{-c 4^k} 2^{kn} (1 + k) \|f\|_{\mathrm{BMO}(\mathbb{R}^n)}.
\]
Summing over $k \geq 0$, we get that
\[
|\tilde I_{\text{far}}(y)| \leq C \|f\|_{\mathrm{BMO}(\R^n)} \sum_{k=0}^\infty e^{-c 4^k} 2^{kn} (1 + k)\le  C \|f\|_{\mathrm{BMO}(\R^n)}.
\]
Combining the near and far estimates, we have
\[
|s\partial_s W_s f(y)| \leq |\tilde I_{\text{near}}(y)| + |\tilde I_{\text{far}}(y)| \leq C \|f\|_{\mathrm{BMO}(\R^n)}.
\]

${\it (2)}$
By   estimate  ${\it (1)}$, we have
\begin{align*}
&\left| |s\partial_s W_s f(x)|^2 - |s\partial_s W_s f(z)|^2 \right| \\ &\leq |s\partial_s W_s f(x) - s\partial_s W_s f(z)| \cdot \left( |s\partial_s W_s f(x)| + |s\partial_s W_s f(z)| \right) \\
&\leq C \|f\|_{\mathrm{BMO}(\R^n)} |s\partial_s W_s f(x) - s\partial_s W_s f(z)|.
\end{align*}
Since $\int_{\R^n}\nabla_\xi K_s(\xi,w)\,dw=0$, the same annular
argument as in part ${\it (1)}$, now using
\[
|\nabla_\xi K_s(\xi,w)|
\leq C s^{-(n+1)/2}
\left(1+\frac{|\xi-w|}{\sqrt{s}}\right)^3
e^{-c|\xi-w|^2/s},
\]
gives the uniform gradient estimate
\[
|\nabla(s\partial_sW_sf)(\xi)|
\leq C s^{-1/2}\|f\|_{\mathrm{BMO}(\R^n)}.
\]
The mean value theorem therefore yields
\[
|s\partial_s W_s f(x)-s\partial_s W_s f(z)|
\leq C\|f\|_{\mathrm{BMO}(\R^n)}
\frac{|x-z|}{\sqrt{s}}.
\]
Combining the preceding two inequalities proves part ${\it (2)}$.
\end{proof}

Now, we continue the proof of Theorem \ref{Thm:gbound}.
For $y,z\in B$, Lemma \ref{lem:BMOest} gives
\begin{align}
|F_r(y)-F_r(z)|
&\leq \int_{4r^2}^{\infty}
 \left||A_sf(y)|^2-|A_sf(z)|^2\right|\frac{ds}{s}\notag\\
&\leq C\|f\|_{\mathrm{BMO}(\R^n)}^2|y-z|
 \int_{4r^2}^{\infty}\frac{ds}{s^{3/2}}
\leq C\|f\|_{\mathrm{BMO}(\R^n)}^2.
\label{eq:tail-oscillation}
\end{align}

Let
\[
m_B=\underset{z\in B} \essinf h(z).
\]
Because $g(f)<\infty$ almost everywhere, $m_B<\infty$. Given
$\varepsilon>0$, choose $z_\varepsilon\in B$, outside the exceptional
null sets, such that
\[
h(z_\varepsilon)\leq m_B+\varepsilon.
\]
For almost every $y\in B$, \eqref{eq:tail-oscillation} implies
\begin{align}
0\leq h(y)-m_B
&=N_r(y)-N_r(z_\varepsilon)
  +F_r(y)-F_r(z_\varepsilon)
  +h(z_\varepsilon)-m_B\notag\\
&\leq N_r(y)+C\|f\|_{\mathrm{BMO}(\R^n)}^2+\varepsilon.
\label{eq:h-minus-inf}
\end{align}
In view of \eqref{eq:local-square-estimate}, this also proves that
$h\in L^1_{\mathrm{loc}}(\R^n)$.

We now verify condition ${\it (iii)}$ of Theorem
\ref{Thm:heatC1.2}. By the normalization of the localized heat kernel,
\[
\frac1{\vartheta_n}\int_BW_{r^2}(x-y)\,dy=1.
\]
Hence, using \eqref{eq:h-minus-inf}, the estimate
$W_{r^2}(x-y)\leq Cr^{-n}$ for $y\in B$, and
\eqref{eq:local-square-estimate}, we obtain
\begin{align*}
&\frac1{\vartheta_n}\int_{B(x,r)}W_{r^2}(x-y)h(y)\,dy
 -\underset{z\in B(x,r)}\essinf h(z)\\
&\quad=\frac1{\vartheta_n}\int_BW_{r^2}(x-y)
 \bigl(h(y)-m_B\bigr)\,dy\\
&\quad\leq \frac{C}{|B|}\int_BN_r(y)\,dy
 +C\|f\|_{\mathrm{BMO}(\R^n)}^2+\varepsilon\\
&\quad\leq C\|f\|_{\mathrm{BMO}(\R^n)}^2+\varepsilon.
\end{align*}
Letting $\varepsilon\downarrow0$ and taking the supremum over
$x\in\R^n$ and $r>0$, we get
\[
\sup_{x\in\R^n,\,t>0}
\left[
\frac1{\vartheta_n}\int_{B(x,\sqrt t)}W_t(x-y)h(y)\,dy
-\underset{z\in B(x,\sqrt t)} \essinf h(z)
\right]
\leq C\|f\|_{\mathrm{BMO}(\R^n)}^2.
\]
Theorem \ref{Thm:heatC1.2}, together with
\eqref{eq:local-equivalence}, now shows that $h=[g(f)]^2$ belongs to
$\mathrm{BLO}(\R^n)$ and
\[
\|[g(f)]^2\|_{\mathrm{BLO}(\R^n)}
\leq C\|f\|_{\mathrm{BMO}(\R^n)}^2.
\]
\end{proof}
Consequently, we obtain the $\BMO-\BLO$ boundedness of the Littlewood-Paley $g$-function.
\begin{cor}
Under the same assumptions as in Theorem \ref{Thm:gbound}, $g(f) \in \mathrm{BLO}(\mathbb{R}^n)$ with
\[
\|g(f)\|_{\mathrm{BLO}(\mathbb{R}^n)} \leq C \|f\|_{\mathrm{BMO}(\mathbb{R}^n)}.
\]
\end{cor}

\begin{proof}
For any ball $B\subset \R^n$,  we have
\begin{multline*}
{1\over |B|}\int_B \left(g(f)(x) - \underset{B}\essinf  g(f)\right) dx \leq {1\over |B|}\int_B \left( [g(f)(x)]^2 - \underset{B}\essinf  [g(f)]^2 \right)^{1/2}dx \\ \leq \left( C \|f\|_{\mathrm{BMO}(\mathbb{R}^n)}^2 \right)^{1/2} = C \|f\|_{\mathrm{BMO}(\mathbb{R}^n)}.
\end{multline*}
Taking supremum over balls $B$, we get the result.
\end{proof}

\section{Proof of Theorem \ref{Thm:heatC} and applications }\label{Sec:1.3}

In this section, we prove Theorem \ref{Thm:heatC}. We then use Theorem \ref{Thm:heatC} to establish the regularity of solutions to the heat equation.

\subsection{Proof of Theorem \ref{Thm:heatC}}
\begin{proof}[Proof of Theorem \ref{Thm:heatC}]
Assume that $f \in \text{BLO}(\R^n)$. We want to show that there exists $C>0$ such that
$$ \abs{W_tf(x)-\underset{z\in B(x,\sqrt t)} {\inf}W_tf(z)}\le C\norm{f}_{{\rm  BLO}(\R^n)},$$  for all $x \in \mathbb{R}^n$ and all $t>0$.
  The heat kernel ${W}_t(x-y)$ is rapidly decreasing and acts as a local average. For a fixed $x$ and $t$, the values ${W}_t f(z)$ for $z$ near $x$ are primarily influenced by the values of $f$ on a ball roughly of radius $r$ with $r\sim \sqrt{t}$.
Take two points $x, z \in \mathbb{R}^n$ with $|x-z| \leq \sqrt{t}$. Let us estimate the difference ${W}_t f(x)-{W}_t f(z)$. Observe that
    \begin{align*}
    {W}_t f(x)-{W}_t f(z) = \int_{\mathbb{R}^n} ({W}_t(x-y)-{W}_t(z-y)) f(y) dy= \int_{\mathbb{R}^n} ({W}_t(x-y)-{W}_t(z-y)) (f(y)-m)dy,
    \end{align*}
for any constant $m$.
The natural choice is to take $m$ to be the essential infimum of $f$ on a ball containing the support of the effective difference of the kernels. Note that $|{W}_t(x-y)-{W}_t(z-y)|$ is significant only when $y$ is in a ball of radius $C\sqrt{t}$ around $x$ (and hence around $z$). Let $B = B(x, \sqrt{t})$ and choose $m=\underset {2B } {\text{ess inf}}\ f$. 
Then,
    \begin{align*}
    &|{W}_t f(x) -{W}_t f(z)|\\
    &\leq \int_{2B} |{W}_t(x-y) - {W}_t(z-y)| |f(y) - \underset {2B } {{\essinf}}\ f| dy\\
    &\quad+\int_{(2B)^c} |{W}_t(x-y) - {W}_t(z-y)| |f(y) - \underset {2B } {{\essinf}}\ f| dy \\
    &=:I+II.
    \end{align*}
    For $I$, since  $\abs{\nabla_x W_t(x-y)}\le C t^{-1/2}W_t(x-y)$ and $|x-z|\le \sqrt t,$ by the mean value theorem, there exists $\xi$ between $x$ and $z$ such that
    \begin{equation*}
    I\le C{|x-z|\over \sqrt t} \int_{2B} W_t(\xi-y)|f(y)-\underset {2B} {{\essinf}}\ f| dy \le {C\over  t^{n/2}} \int_{2B}|f(y)-\underset {2B} {{\essinf}}\ f| dy\le C\norm{f}_{{\rm BLO}(\R^n)}.
    \end{equation*}
    For $II$, it follows from the definition of $W_t(x-y)$ that
    \begin{align*}
    II &\leq C{|x-z|\over \sqrt t} \int_{(2B)^c} W_t(\xi-y)|f(y)-\underset {2B} {{\essinf}}\ f| dy\\
    &\le C \sum_{k=1}^{+\infty}\int_{2^{k+1}B \backslash 2^{k}B} W_t(\xi-y)|f(y)-\underset {2B} {{\essinf}}\ f| dy\\
    &\le C t^{-n/2} \sum_{k=1}^{+\infty}e^{-c4^k}\int_{2^{k+1}B } |f(y)-\underset {2^{k+1}B} {{\essinf}}\ f| dy\\
    &\le  C t^{-n/2}\norm{f}_{{\rm BLO}(\R^n)} \sum_{k=1}^{+\infty}e^{-c4^k}\cdot (k+1)\cdot  \abs{2^{k+1}B }\\
    &\le C \norm{f}_{{\rm BLO}(\R^n)} \sum_{k=1}^{+\infty}e^{-c4^{k-1}}k 2^{kn} \\
    &\le C \norm{f}_{{\rm BLO}(\R^n)}.
    \end{align*}
    Combining the estimates of $I$ and $II$, we prove that
    $$|{W}_t f(x) -{W}_t f(z)|\le C\norm{f}_{{\rm BLO}(\R^n)}.$$
Since the inequality above holds for any $z \in B(x, \sqrt{t})$,
$$W_tf(x)-\underset {z\in B(x,\sqrt t)}{{\inf}} W_tf(z)\le C\norm{f}_{{\rm BLO}(\R^n)},$$
which is the desired result. Taking the supremum over $x$ and $t$ finishes the proof.
\end{proof}

\subsection{Regularity  and oscillation estimate of the solutions to the heat equation}

We can use the heat semigroup inequality for functions in BLO space as in Theorem \ref{Thm:heatC} to prove a regularity  of the solutions to the heat equation with BLO boundary values.

\begin{thm}\label{Thm:AppParabolic}
Let $f \in \mathrm{BLO}(\mathbb{R}^n)$.
Then $u(x,t)=W_t f(x)$ satisfies the following regularity estimate
$$
u(x,t) \leq \inf_{z \in B(x,\sqrt{t})} u(z,t) + C \|f\|_{\mathrm{BLO}(\R^n)}\quad \text{for all } x \in \mathbb{R}^n,\ t > 0,
$$
where $C > 0$ depends only on the dimension $n$.
\end{thm}

\begin{proof}
The solution is given by the heat semigroup
$$
u(x,t) = W_t f(x) = \int_{\mathbb{R}^n} (4\pi t)^{-n/2} e^{-\frac{|x-y|^2}{4t}} f(y)\, dy.
$$
Since $f \in \mathrm{BLO}(\mathbb{R}^n)$, by Theorem \ref{Thm:heatC} we have
$$
\sup_{t > 0} \left\| W_t f(\cdot) - \underset{z \in B(\cdot,\sqrt{t})}{{\inf}}\, W_t f(z) \right\|_{L^\infty(\mathbb{R}^n)} \leq C \|f\|_{\mathrm{BLO}(\R^n)}.
$$
In particular, for any fixed $t > 0$ and a.e.  $x \in \mathbb{R}^n$, we have
$$
W_t f(x) - \inf_{z \in B(x,\sqrt{t})} W_t f(z) \leq C \|f\|_{\mathrm{BLO}(\R^n)} .
$$
Rewriting this in terms of $u(x,t)$, we obtain
$$
u(x,t) \leq \inf_{z \in B(x,\sqrt{t})} u(z,t) + C \|f\|_{\mathrm{BLO}(\R^n)},
$$
which is the desired estimate.
\end{proof}

By the theorem above, we can get an oscillation estimate of solutions to the heat equation with $\mathrm{BLO}$ boundary values in the following.
\begin{thm}\label{Thm:Harnack}
Under the same assumptions as in Theorem \ref{Thm:AppParabolic}, for any $x_0 \in \mathbb{R}^{n}$ and $t > 0$, we have
\[
\sup_{x\in B(x_0,\sqrt{t})} u(x,t) - \inf_{x\in B(x_0,\sqrt{t})} u(x,t) \leq C\|f\|_{\mathrm{BLO}(\mathbb{R}^{n})}.
\]
\end{thm}

\begin{proof}
For any $x_0 \in \mathbb{R}^{n}$ and $t > 0$, denote $B = B(x_0, \sqrt{t})$. For any $x, y \in B$, let $z$ be the midpoint of $x$ and $y$.  Therefore, $x \in B(z, \sqrt{t})$ and $y \in B(z, \sqrt{t})$.

By Theorem \ref{Thm:AppParabolic}, we have
\[
u(z,t) \leq \inf_{w \in B(z,\sqrt{t})} u(w,t) + C\|f\|_{\mathrm{BLO}(\mathbb{R}^{n})}.
\]
Since $x \in B(z,\sqrt{t})$, we have $\displaystyle \inf_{w \in B(z,\sqrt{t})} u(w,t) \leq u(x,t)$. So
\begin{equation}\label{eq:1}
u(z,t) \leq u(x,t) + C\|f\|_{\mathrm{BLO}(\mathbb{R}^{n})}.
\end{equation}

Moreover, we apply Theorem \ref{Thm:AppParabolic} to the point $x$ and get
\[
u(x,t) \leq \inf_{w \in B(x,\sqrt{t})} u(w,t) + C\|f\|_{\mathrm{BLO}(\mathbb{R}^{n})}.
\]
Since $z \in B(x,\sqrt{t})$, we have $\underset{w \in B(x,\sqrt{t})}\inf u(w,t) \leq u(z,t)$. So
\begin{equation}\label{eq:2}
u(x,t) \leq u(z,t) + C\|f\|_{\mathrm{BLO}(\mathbb{R}^{n})}.
\end{equation}
From inequalities \eqref{eq:1} and \eqref{eq:2}, we obtain
\[
|u(x,t) - u(z,t)| \leq C\|f\|_{\mathrm{BLO}(\mathbb{R}^{n})}.
\]
Similarly, by applying Theorem \ref{Thm:AppParabolic} to $y$ and $z$, we get
\[
|u(y,t) - u(z,t)| \leq C\|f\|_{\mathrm{BLO}(\mathbb{R}^{n})}.
\]
Therefore,
\[
|u(x,t) - u(y,t)| \leq |u(x,t) - u(z,t)| + |u(z,t) - u(y,t)| \leq C\|f\|_{\mathrm{BLO}(\mathbb{R}^{n})}.
\]
Since this inequality holds for all $x, y \in B$, we conclude
\[
\sup_{x,y \in B} |u(x,t) - u(y,t)| \leq C\|f\|_{\mathrm{BLO}(\mathbb{R}^{n})},
\]
which implies
\[
\sup_{B} u(x,t) - \inf_{B} u(x,t) \leq C\|f\|_{\mathrm{BLO}(\mathbb{R}^{n})}.
\]
This completes the proof of the oscillation estimate.
\end{proof}

\section{Proof of Theorem \ref{HeatCbyWeight}}\label{Sec3}
In this section, we will give the proof of Theorem \ref{HeatCbyWeight}.
\begin{proof}
\text{(\it i)  $\Leftrightarrow$ (ii).} This is  the classical result in \cite{CR} which was developed by  Coifman and Rochberg.

\text{(\it ii) $\Rightarrow$ (iii).} Assume that there exists $\varepsilon > 0$ such that $e^{\varepsilon f} \in A_1$ with constant $0<C_1$ as in \eqref{eq:1.11}. The heat kernel is  an approximation of the identity, and there exists a constant $C_2 > 0$ such that for any $t > 0$,
\[
W_t(e^{\varepsilon f})(x) \leq C_2 M(e^{\varepsilon f})(x) \quad \text{for any } x \in \mathbb{R}^n.
\]
Combining with the $A_1$ condition, we get
\[
W_t(e^{\varepsilon f})(x) \leq C_2 M(e^{\varepsilon f})(x) \leq C_2 C_1 e^{\varepsilon f(x)}.
\]
Thus, $(iii)$ holds with $C = C_2 C_1$.

\text{(\it iii) $\Rightarrow$ (ii).}
Assume $(iii)$: there exist $\varepsilon > 0$ and $C > 0$ such that for all $t > 0$ and almost every $x \in \R^n$,
\[
W_t(e^{\varepsilon f})(x) \leq C e^{\varepsilon f(x)}.
\]
We will show that this implies $e^{\varepsilon f} \in A_1$.
Recall that the Hardy--Littlewood maximal function $M$ satisfies the pointwise inequality
\[
M(g)(x) \leq C_n \sup_{t > 0} W_t(g)(x) \qquad \text{for any nonnegative function } g,
\]
where $C_n$ is a dimensional constant. This follows because the heat kernel is radial, decreasing, and integrable, and such kernels dominate the characteristic function of balls up to a constant.
Applying this to $g = e^{\varepsilon f}$, we obtain
\[
M(e^{\varepsilon f})(x) \leq C_n \sup_{t > 0} W_t(e^{\varepsilon f})(x) \leq C_n\cdot C e^{\varepsilon f(x)}.
\]
Thus $e^{\varepsilon f} \in A_1$ with  constant at most $C_n\cdot C$. This proves $(ii).$
 This completes the proof of  Theorem \ref{HeatCbyWeight}.
\end{proof}

Theorem \ref{HeatCbyWeight}  also yields that we can get a quantitative version inequality as follows.
\begin{prop}\label{prop:quantitative}
There exist constants $c, C > 0$ depending only on the dimension $n$ such that for any $f \in \mathrm{BLO}(\mathbb{R}^n)$,
\[
c\|f\|_{\mathrm{BLO}(\mathbb{R}^n)} \leq N(f) \leq C\|f\|_{\mathrm{BLO}(\mathbb{R}^n)},
\]
where \[
N(f) := \inf \left\{ \frac{1}{\varepsilon} \log C_0 : \varepsilon > 0, \; C_0 > 1, \; \sup_{t>0} W_t(e^{\varepsilon f}) \leq C_0 e^{\varepsilon f} \text{ a.e.} \right\}.
\]
\end{prop}

\begin{proof}
Let us prove the upper bound $N(f) \leq C \|f\|_{\mathrm{BLO}(\mathbb{R}^n)}$ first.
Let $f \in \mathrm{BLO}(\mathbb{R}^n)$. By Bennett's lemma (Lemma 2 in [2]), there exist dimensional constants $c_1, c_2 > 0$ such that
\[
c_1 \|f\|_{\mathrm{BLO}(\mathbb{R}^n)} \leq \|\widetilde{M}f - f\|_{L^\infty(\mathbb{R}^n)} \leq c_2 \|f\|_{\mathrm{BLO}(\mathbb{R}^n)},
\]
where $\displaystyle \widetilde{M}f(x) = \sup_{B \ni x} \frac{1}{|B|} \int_B f(y) \, dy$.
Set $K := \|\widetilde{M}f - f\|_{L^\infty(\mathbb{R}^n)}$. Then for any ball $B = B(x, r)$,
\[
\frac{1}{|B|} \int_B f(y) \, dy \leq f(x) + K.
\]
Since $\mathrm{BLO}(\mathbb{R}^n) \subset \mathrm{BMO}(\mathbb{R}^n)$, we have $f \in \mathrm{BMO}(\mathbb{R}^n)$ with
\[
\|f\|_{\mathrm{BMO}(\mathbb{R}^n)} \leq 2 \|f\|_{\mathrm{BLO}(\mathbb{R}^n)}.
\]
Indeed, for any ball $B$,
\begin{multline*}
\frac{1}{|B|} \int_B |f - f_B| \leq \frac{1}{|B|} \int_B |f - \underset{ B} {\essinf}f| + |\underset{ B} {\essinf}f- f_B| \\\leq \|f\|_{\mathrm{BLO}(\mathbb{R}^n)} + \frac{1}{|B|} \int_B (f - \underset{ B} {\essinf} f) \leq 2 \|f\|_{\mathrm{BLO}(\mathbb{R}^n)}.
\end{multline*}
By the John--Nirenberg inequality for $\mathrm{BMO}$ functions, there exist constants $c_n, C_n > 0$ depending only on $n$ such that for any ball $B$,
\[
\frac{1}{|B|} \int_B \exp\left( \frac{c_n |f - f_B|}{\|f\|_{\mathrm{BMO}(\mathbb{R}^n)}} \right) \leq C_n.
\]
Choose $\varepsilon = \min\left( \frac{c_n}{2 \|f\|_{\mathrm{BLO}(\mathbb{R}^n)}}, \, \frac{1}{K} \right)$.
Because $\|f\|_{\mathrm{BMO}(\mathbb{R}^n)} \leq 2 \|f\|_{\mathrm{BLO}(\mathbb{R}^n)}$, we have $$\frac{c_n}{2 \|f\|_{\mathrm{BLO}(\mathbb{R}^n)}} \leq \frac{c_n}{\|f\|_{\mathrm{BMO}(\mathbb{R}^n)}}.$$ So $\varepsilon \leq \frac{c_n}{\|f\|_{\mathrm{BMO}(\mathbb{R}^n)}}$, and $\varepsilon K \leq 1$.
Now fix a ball $B = B(x, r)$. Since $f_B \leq f(x) + K$, we obtain
\[
\begin{aligned}
\frac{1}{|B|} \int_B e^{\varepsilon f(y)} \, dy
&\leq \frac{1}{|B|} \int_B e^{\varepsilon (f(y) - f_B)} e^{\varepsilon f_B} \, dy
\leq e^{\varepsilon f_B} \cdot \frac{1}{|B|} \int_B e^{\varepsilon |f(y) - f_B|} \, dy \\
&\leq e^{\varepsilon f_B} \cdot \frac{1}{|B|} \int_B \exp\left( \frac{c_n |f(y) - f_B|}{\|f\|_{\mathrm{BMO}}} \right) dy
\leq e^{\varepsilon f_B} \, C_n \\
&\leq C_n e^{\varepsilon (f(x) + K)}
= C_n e^{\varepsilon K} e^{\varepsilon f(x)} \leq C_n e \, e^{\varepsilon f(x)}.
\end{aligned}
\]
Taking the supremum over all balls containing $x$ yields
\[
 M(e^{\varepsilon f})(x) \leq C_n e \, e^{\varepsilon f(x)}.
\]
Hence, $e^{\varepsilon f} \in A_1$ with $A_1$ constant at most $C_n e$.
A standard comparison gives a constant $C'_n > 0$ (depending only on $n$) such that for any nonnegative function $g$,
\[
W_t g(x) \leq C'_n  M g(x) \quad \text{for all } t > 0, \; x \in \mathbb{R}^n.
\]
Applying this to $g = e^{\varepsilon f}$, we get
\[
\sup_{t > 0} W_t(e^{\varepsilon f})(x) \leq C'_n M(e^{\varepsilon f})(x) \leq C'_n C_n e \, e^{\varepsilon f(x)}.
\]
Thus, we have found $\varepsilon > 0$ and $C_0 = C'_n C_n e$ such that
\[
\sup_{t > 0} W_t(e^{\varepsilon f})(x) \leq C_0 e^{\varepsilon f}(x) \quad \text{a.e.}\ \  x\in \mathbb{R}^n.
\]
Note that $\varepsilon \geq \min\left( \frac{c_n}{2 \|f\|_{\mathrm{BLO}(\mathbb{R}^n)}}, \, \frac{1}{c_2 \|f\|_{\mathrm{BLO}(\mathbb{R}^n)}} \right)$ because $K \leq c_2 \|f\|_{\mathrm{BLO}(\mathbb{R}^n)}$.
Hence, there exists a constant $c' > 0$ (depending only on $n$) such that $\displaystyle \varepsilon \geq \frac{c'}{\|f\|_{\mathrm{BLO}(\mathbb{R}^n)}}$. Consequently,
\[
\frac{1}{\varepsilon} \log C_0 \leq \frac{\|f\|_{\mathrm{BLO}(\mathbb{R}^n)}}{c'} \log(C'_n C_n e) = C \|f\|_{\mathrm{BLO}(\mathbb{R}^n)}.
\]
Therefore, we get that  $$N(f) \leq C \|f\|_{\mathrm{BLO}(\mathbb{R}^n)}.$$

Now, we prove the lower bound $c \|f\|_{\mathrm{BLO}(\mathbb{R}^n)} \leq N(f)$.
 Assume that for some \(\varepsilon>0\) and \(C_0>1\) we have
\begin{equation}\label{assumption}
\sup_{t>0}W_t(e^{\varepsilon f})(x)\le C_0e^{\varepsilon f(x)}\qquad\text{for a.e. }x\in\R^n.
\end{equation}
Let \( B \) be a ball of radius \( r \). Set \( t = r^2 \). For  a.e.  \( x \in B \),  there exists \(0<c_n<1\) such that for all \( y \in B \),
\[
(4\pi t)^{-n/2} e^{-|x-y|^2/(4t)} \geq c_n |B|^{-1},
\]
since \( |x-y| \leq 2r = 2\sqrt{t} \).
Then for any \( x \in B \),
\begin{equation}\label{7}
W_t(e^{\varepsilon f})(x)
\geq \int_B (4\pi t)^{-n/2} e^{-|x-y|^2/(4t)} e^{\varepsilon f(y)} dy\\
\geq c_n |B|^{-1} \int_B e^{\varepsilon f(y)} dy.
\end{equation}
Combining \eqref{assumption} and \eqref{7} gives
\begin{align}\label{8}
c_n |B|^{-1} \int_B e^{\varepsilon f(y)} dy \leq C_0 e^{\varepsilon f(x)} \quad \text{a.e.}\ \  x \in B.
\end{align}
Let \( a_B = \essinf_B f \). Then \eqref{8} implies
\begin{align}\label{9}
c_n |B|^{-1} \int_B e^{\varepsilon (f(y) - a_B)} dy \leq C_0 e^{\varepsilon (f(x) - a_B)} \quad \text{a.e.}\ \  x \in B.
\end{align}
Since \( a_B \) is the essential infimum of $f$ over $B$, for any \( \delta > 0 \) there exists $B_\delta\subset B$ such that $|B_\delta|>0$, and for any \( x_\delta \in B_\delta \) we have  \( f(x_\delta) - a_B \leq \delta \). Applying \eqref{9} to \( x_\delta \),
\[
c_n |B|^{-1} \int_B e^{\varepsilon (f(y) - a_B)} dy \leq C_0 e^{\varepsilon \delta}.
\]
Letting \( \delta \to 0^+ \) yields
\[
|B|^{-1} \int_B e^{\varepsilon (f(y) - a_B)} dy \leq \frac{C_0}{c_n}.
\]
By Jensen's inequality,
\[
\exp\left( \varepsilon |B|^{-1} \int_B (f(y) - a_B) dy \right) \leq |B|^{-1} \int_B e^{\varepsilon (f(y) - a_B)} dy \leq \frac{C_0}{c_n}.
\]
Taking logarithms,
\[
\varepsilon |B|^{-1} \int_B (f(y) - a_B) dy \leq \log\left( \frac{C_0}{c_n} \right).
\]
Hence,
\[
|B|^{-1} \int_B (f(y) - a_B) dy \leq \frac{1}{\varepsilon} \log\left( \frac{C_0}{c_n} \right).
\]
Since this holds for every ball \( B \),
\[
\|f\|_{\BLO(\mathbb{R}^n)} \leq \frac{1}{\varepsilon} \log\left( \frac{C_0}{c_n} \right).
\]
The pair \( (\varepsilon, C_0) \) is arbitrary satisfying \eqref{assumption}. Taking the infimum over all such pairs, we obtain
\[
\|f\|_{\BLO(\mathbb{R}^n)} \leq \frac{1}{c} N(f),
\]
where \( c > 0 \) depends only on \( n \). More precisely, \( N(f) \geq c_n' \|f\|_{\BLO(\mathbb{R}^n)} \) with \( c_n' = c_n / \log(C_0) \).
\end{proof}

\section{Non-linearity of ${\rm BLO}(\mathbb R^n)$}\label{Sec:4}

In this section, we give a simple example of functions in ${\rm BLO}(\mathbb R^n)$ with $n=1$.

\begin{prop}\label{p-BLO examp}
Let \begin{align*}
f(x)=
\begin{cases}
-\ln|x|,&\ x\ne 0,\\
0,&\ x=0.
\end{cases}
\end{align*} Then $f \in {\rm BLO}(\mathbb R)$
while $-f\notin {\rm BLO}(\mathbb R)$.

\end{prop}

\begin{proof} For any $I=(a, b)$ with $-\infty<a<0<b<\infty$,
$$
\underset{I} {\essinf} \ (-f)=-\infty.$$
Thus $-f \notin {\rm BLO}(\mathbb R)$. Now we prove $f \in {\rm BLO}(\mathbb R)$. To this end, it suffices to show that
there exists a positive constant $C$ such that for any interval $I$,
\begin{equation}\label{e-blo defn}
f_I-\underset{I} {\essinf} \ f\le C.
\end{equation}

For any given interval $I=(a, b)$,  we consider the following  three cases.
\begin{itemize}
  \item [(i)] $0 \leq a<b<\infty$. In this case,
$$
f_I  =\frac{1}{|I|} \int_a^b-\ln x d x
=\frac{-1}{b-a}[b \ln b-a \ln a-(b-a)]  =1-\frac{b \ln b-a \ln a}{b-a}.
$$
Since $\underset{I} {\essinf} \ f=f(b)=-\ln b$, observe that if $a=0$, then
$$ f_I-\underset{I} {\essinf} \ f=1.$$
On the other hand, for $a>0$, by $b/a>1$, we conclude that
 \begin{align*}
f_I-\underset{I} {\essinf} \ f&=1-\frac{b \ln b-a \ln a-(b-a)\ln b}{b-a} =1-\frac{a(\ln b-\ln a)}{b-a} =1-\frac{\ln \frac b a}{b / a-1}<1.
\end{align*}
  \item [(ii)]  $a<b\le 0$. In this case, by similarity we also have that
$$ f_I-\underset{I} {\essinf} \ f\le 1.$$
  \item  [(iii)] $a<0<b$. Observe that
\begin{align*}
f_I&=\frac{1}{b-a}\left[\int_a^0+\int_0^b\right](-\ln |x|) d x =\frac{1}{b-a}[a\ln (-a)-a-b \ln b+b]  =1-\frac{b \ln b-a \ln (-a)}{b-a}.
\end{align*}
If $|a|>b$, then $\underset{I} {\essinf} \ f=f(a)=-\ln (-a)$, which implies

\begin{align*}
f_I-\underset{I} {\essinf} \ f & =1-\frac{b \ln b-a \ln (-a)-(b-a)\ln (-a)}{b-a}  =1-\frac{b \ln \frac{b}{-a}}{b-a}  =1+\frac{\ln \frac{-a}{b}}{1 +\frac{-a}{b}}\leq 2.
\end{align*}
Similarly, if $|a| \leq b$, then $\underset{I} {\essinf} \ f=f(b)=-\ln b.$ This yields
\begin{align*}
f_I-\underset{I} {\essinf} \ f& =1-\frac{b \ln b-a \ln (-a)-(b-a)\ln b}{b-a}
=1+\frac{\ln \frac{b}{-a}}{\frac{b}{-a}+1}
 \leq 2.
\end{align*}
\end{itemize}

By summarizing the three cases above, we conclude that \eqref{e-blo defn} holds, which completes the proof of Proposition \ref{p-BLO examp}.
\end{proof}

The above example in the proposition shows that the $\mathrm{BLO}$ space is not a linear space. In particular,  when $f,g\in \mathrm{BLO}(\R^n)$, $-f$  and $f-g$ may not in $\mathrm{BLO}(\R^n)$.

\vspace{1em}

\noindent{\bf \Large Declarations}

\vspace{1em}

\noindent{\bf Ethical Approval.} Not applicable.

\vspace{1em}
\noindent{\bf Competing interests.}
The authors declare that there is no conflict of interest regarding the publication of this article.

\vspace{1em}
\noindent{\bf Authors' contributions.} The authors contributed equally on this manuscript.



\vspace{3em}

\end{document}